\documentclass[11pt]{amsart}
\usepackage{amssymb}
\usepackage{amsmath}
\usepackage[hidelinks]{hyperref}
\usepackage{etoolbox}
\usepackage{enumitem}
\usepackage{xcolor}
\usepackage{tikz-cd}
\usepackage[normalem]{ulem}
\hypersetup{
	colorlinks,
	linkcolor={red!30!black},
	citecolor={blue!50!black},
	urlcolor={blue!80!black}
}

\newtheorem{theorem}{Theorem}[section]
\newtheorem{lemma}[theorem]{Lemma}

\newtheorem{proposition}[theorem]{Proposition}

\newtheorem{conjecture}[theorem]{Conjecture}

\theoremstyle{remark}
\newtheorem{remark}[theorem]{Remark}

\newcommand{\p}{\vskip .4cm}

\newcommand{\F}{\mathbb{F}}
\newcommand{\Z}{\mathbb{Z}}
\newcommand{\Q}{\mathbb{Q}}

\newcommand{\C}{\mathbb{C}}
\newcommand{\Cc}{\mathbb{C}^{\times}}
\newcommand{\f}{\mathfrak{f}}

\newcommand{\cO}{\mathcal{O}}

\newcommand{\til}{\tilde}

\newcommand{\ra}{\rightarrow}

\newcommand{\Lie}{\operatorname{Lie}}

\newcommand{\Ad}{\operatorname{Ad}}

\newcommand{\Frob}{\mathrm{Frob}}

\newcommand{\Ind}{\mathrm{Ind}}
\newcommand{\Tr}{\operatorname{Tr}}

\newcommand{\Id}{\operatorname{Id}}
\newcommand{\fg}{\mathfrak{g}}

\newcommand{\fm}{\mathfrak{m}}

\renewcommand{\span}{\operatorname{span}}

\newcommand{\matr}[1]{\left[\begin{matrix}#1\end{matrix}\right]}

\newcommand{\supp}{\operatorname{supp}}

\newcommand{\WF}{\operatorname{WF}}
\newcommand{\bWF}{\overline{\operatorname{WF}}}
\newcommand{\fixme}[1]{{\color{red}#1}}

\newcommand{\Sp}{\mathrm{Sp}}
\newcommand{\PSp}{\mathrm{PSp}}

\begin{document}

\title{On two definitions of wave-front sets for $p$-adic groups}
\author{Cheng-Chiang Tsai}
\email{chchtsai@gate.sinica.edu.tw}
\address{Institute of Mathematics, Academia Sinica, 6F, Astronomy-Mathematics Building, No. 1, Sec. 4, Roosevelt Road, Taipei, Taiwan,\vskip.05cm
\noindent also Department of Applied Mathematics, National Sun Yat-Sen University, and Department of Mathematics, National Taiwan University}
\thanks{The author is supported by NSTC grants 110-2115-M-001-002-MY3, 113-2115-M-001-002 and 113-2628-M-001-012}
\begin{abstract}
	The wave-front set for an irreducible admissible representation of a $p$-adic reductive group is the set of maximal nilpotent orbits which appear in the local character expansion. By M\oe glin-Waldspurger, they are also the maximal nilpotent orbits whose associated degenerate Whittaker models are non-zero. However, in the literature there are two versions commonly used, one defining maximality using analytic closure and the other using Zariski closure. We show that these two definitions are inequivalent for $G=\Sp_4$.
\end{abstract}
\makeatletter
\patchcmd{\@maketitle}
{\ifx\@empty\@dedicatory}
{\ifx\@empty\@date \else {\vskip3ex \centering\footnotesize\@date\par\vskip1ex}\fi
	\ifx\@empty\@dedicatory}
{}{}
\patchcmd{\@adminfootnotes}
{\ifx\@empty\@date\else \@footnotetext{\@setdate}\fi}
{}{}{}
\makeatother

\maketitle 

\vskip-1cm $\;$

\tableofcontents

\vskip-1cm $\;$

\section{Introduction}

Let $F$ be a finite extension of $\Q_p$ and $G$ be a connected reductive group over $F$. Write $\fg:=\Lie G$. The local character expansion of Howe and Harish-Chandra \cite[Thm. 16.2]{HC} asserts that for any irreducible admissible $\C$-representation $\pi$ of $G(F)$, there exist constants $c_{\cO}(\pi)\in\C$ indexed by nilpotent $\Ad(G(F))$-orbits $\cO\subset\fg(F)$, together with a neighborhood $U=U_{\pi}$ of $0\in\fg(F)$, such that the character $\Theta_{\pi}$ of $\pi$ satisfies the following identity of distributions on $U$:
\begin{equation}\label{LCE}
	(\Theta_{\pi}\circ\log^*)|_U\equiv\sum_{\cO}c_{\cO}(\pi)\hat{I}_{\cO}|_U.
\end{equation}
Here $I_{\cO}$ is the distribution of integrating a function on $\cO$ with any $G(F)$-invariant positive measure, and $\hat{I}_{\cO}$ its Fourier transform, namely $\hat{I}_{\cO}(f):=I_{\cO}(\hat{f})$.

In \cite{MW87}, M\oe glin and Waldspurger generalized a result of Rodier \cite{Rod75} and showed that for $\cO\in\max\{\cO\;|\;c_{\cO}(\pi)\not=0\}$, the quantity $c_{\cO}(\pi)$ with suitable normalization is equal to the dimension of the degenerate Whittaker model for $\pi$ relative to $\cO$. Degenerate Whittaker models are local analogues and necessary conditions for existence of Fourier coefficients for automorphic forms. The set $\max\{\cO\;|\;c_{\cO}(\pi)\not=0\}$ is therefore of particular interest, and is typically called the {\bf wave-front set}. However, there are two partial orders commonly used in the literature: For two nilpotent $\Ad(G(F))$-orbits $\cO_1$ and $\cO_2$ the partial order $\cO_1<\cO_2$ is defined either (i) if the analytic closure (using the Hausdorff $p$-adic topology on $\fg(F)$) of $\cO_1$ is strictly contained in the analytic closure of $\cO_2$, or alternatively (ii) if the Zariski closure of $\cO_1$ is strictly contained in the Zariski closure of $\cO_2$.

Let us denote by $\WF^{rat}(\pi):=\max\{\cO\;|\;c_{\cO}(\pi)\not=0\}$ given by the first definition, and $\WF^{Zar}(\pi)$ the analogous set given by the second definition. Since the Zariski closure is larger than the analytic closure, we have an obvious inclusion $\WF^{rat}(\pi)\supseteq \WF^{Zar}(\pi)$. At the same time, there is the notion of geometric wave-front sets: Fix an algebraic closure $\bar{F}$ of $F$ and let $\bWF^{rat}(\pi)$ (resp. $\bWF^{Zar}(\pi)$) be the set of $\Ad(G(\bar{F}))$-orbits in $\fg(\bar{F})$ that meet those in $\WF^{rat}(\pi)$ (resp. $\WF^{Zar}(\pi)$). Again we have $\bWF^{rat}(\pi)\supseteq \bWF^{Zar}(\pi)$. We note that by \cite[Prop. 3.5.75]{Poo17}, any $G(F)$-orbit $\cO\subset\fg(F)$ is Zariski dense in the $G(\bar{F})$-orbit it sits in. Hence $\bWF^{Zar}(\pi)$ is equal to the set of maximal geometric orbits that appear in (\ref{LCE}). We thank Emile Okada for clarifying this.

The set $\WF^{rat}(\pi)$ was used in \cite{MW87}, \cite{Moe96} and \cite{GGS21}, and many others. On the other hand, $\bWF^{Zar}(\pi)$ was used in for example \cite{Wal18}. Both $\WF^{rat}(\pi)$ and $\bWF^{Zar}(\pi)$ were discussed in \cite{CMBO23} while their main results determine $\bWF^{Zar}(\pi)$ but not $\bWF^{rat}(\pi)$. Nevertheless, in \cite{JLZ22a} the main conjecture 
\cite[Conj. 1.3]{JLZ22a} is stated for $\WF^{Zar}(\pi)$ but it seems that the spirit might work for $\WF^{rat}(\pi)$ as well. Given the abundance of results on the topic, it is desirable to know how/whether $\WF^{rat}(\pi)$ and $\WF^{Zar}(\pi)$ (resp. $\bWF^{rat}(\pi)$ and $\bWF^{Zar}(\pi)$) could be different. In fact, the long-standing conjecture about geometric wave-front sets, proposed and proved for $\mathrm{GL}_n$ in \cite{MW87}, asserted that:

\begin{conjecture}\label{conj:geom} For any irreducible admissible representation $\pi$ of $G(F)$, the set $\bWF^{rat}(\pi)$ is a singleton.
\end{conjecture}

Since $\bWF^{Zar}(\pi)$ is obviously non-empty, the validity of Conjecture \ref{conj:geom} for any $\pi$ is equivalent to the validities of the following two statements:

\begin{conjecture}\label{conj:geom'} (Counterexample in \cite[Thm. 1.1]{Tsa24}) $\bWF^{Zar}(\pi)$ is a singleton.
\end{conjecture}

\begin{conjecture}\label{conj:equivdef} We have $\bWF^{rat}(\pi)=\bWF^{Zar}(\pi)$ or equivalently $\WF^{rat}(\pi)=\WF^{Zar}(\pi)$.
\end{conjecture}

As indicated above, the first counterexample for Conjecture \ref{conj:geom} is a counterexample to Conjecture \ref{conj:geom'}. The purpose of this paper is to show that Conjecture \ref{conj:equivdef} also has a counterexample, in fact in the case of split rank 2, which is the smallest absolute rank where Conjecture \ref{conj:equivdef} becomes non-trivial.

Let $p\ge 11$ be any prime number, $q$ a power of $p$ with $q\equiv 1\operatorname{mod}4$, and $F$ any non-archimedean local field with residue field $\F_q$ and fixed uniformizer $\varpi\in F$. Let $G=\Sp_4/_F$ be the group of linear operators on $F^4$ that preserve the symplectic form
\begin{equation}\label{form1}
\langle \vec{x},\vec{y}\rangle=x_1y_4+x_2y_3-x_3y_2-x_4y_1.
\end{equation}
Denote by $\fm$ the maximal ideal and $\fm^0=\cO_F$ the ring of integers in $F$. Consider the Moy-Prasad filtration $(G(F)_r)_{r\in\frac{1}{2}\Z_{\ge 0}}$ ``associated to the Siegel parahoric.'' It is given by
\[
G(F)_n:=\{g\in G(F)\;|\;g-\Id_4\in
\matr{
	\fm^n&\fm^n&\fm^n&\fm^n\\
	\fm^n&\fm^n&\fm^n&\fm^n\\
	\fm^{n+1}&\fm^{n+1}&\fm^n&\fm^n\\
	\fm^{n+1}&\fm^{n+1}&\fm^n&\fm^n	
}
\}\text{, and}
\]
\begin{equation}\label{MP}
G(F)_{n+\frac{1}{2}}:=\{g\in G(F)\;|\;g-\Id_4\in
\matr{
	\fm^{n+1}&\fm^{n+1}&\fm^n&\fm^n\\
	\fm^{n+1}&\fm^{n+1}&\fm^n&\fm^n\\
	\fm^{n+1}&\fm^{n+1}&\fm^{n+1}&\fm^{n+1}\\
	\fm^{n+1}&\fm^{n+1}&\fm^{n+1}&\fm^{n+1}	
}
\}
\end{equation}
for $n\in\Z_{\ge 0}$. The group $G(F)_1$ is a normal subgroup of $G(F)_{1/2}$, and the quotient may be identified as
\[
V:=G(F)_{\frac{1}{2}}/G(F)_1\cong\{\matr{0&0&b&a\\0&0&c&b\\e&d&0&0\\f&e&0&0}\;|\;a,b,c\in\cO_F/\fm,\;d,e,f\in \fm/\fm^2\}.
\]
Fix an additive character $\psi:F\ra\Cc$ that is trivial on $\fm$ but non-trivial on $\cO_F$. Consider
\begin{equation}\label{A}
A:=\matr{
	0&0&0&\varpi^{-1}\\
	0&0&\varpi^{-1}&0\\
	1&0&0&0\\
	0&1&0&0
}\in\fg(F).
\end{equation}
Denote by $\psi_A:V\ra\Cc$ the character $B\mapsto \psi(\Tr(AB))$, and $\til{\psi}_A$ its pullback to $G(F)_{1/2}$. Conjecture \ref{conj:equivdef} is disproved by:

\begin{theorem}\label{main}
	For any irreducible component $\pi$ of the compact induction
	\[
	\operatorname{c-ind}_{G(F)_{\frac{1}{2}}}^{G(F)}\til{\psi}_{A}
	\]
	we have that $\WF^{rat}(\pi)$ contains two regular nilpotent orbits and also a subregular nilpotent orbit. Consequently $\WF^{Zar}(\pi)$ contains only the two regular nilpotent orbits.
\end{theorem}

In fact, the subregular orbit is the unique one not contained in the analytic closure of the previous two regular nilpotent orbits.
The representation $\pi$ is one of the so-called epipelagic representations in \cite{RY14}. Prior to this work, similar representations for much higher rank groups had already been studied in a joint work in progress of Chi-Heng Lo and the author to produce a counterexample to Conjecture \ref{conj:geom'} for split groups (rather than for ramified groups as in \cite{Tsa24}). We also remark that in the language of the newer paper \cite{Tsa23b}, we have $\WF^{rat}(\pi)=\WF^{rat}(A)$ and the result may well be interpreted as for the wave-front set of $A\in\fg(F)$.

\subsection*{Acknowledgment} I wish to thank Chi-Heng Lo, Lei Zhang, Emile Okada and Alexander Bertoloni Meli for very helpful conversations. I am grateful to National University of Singapore for a wonderful environment and their hospitality during a visit. I am thankful to the referee for helpful suggestions. Lastly I thank ChatGPT for polishing some English writings.\p

\section{Nilpotent orbits}

For our $G=\Sp_4$, the subregular nilpotent $\Ad(G(F))$-orbits correspond to partition $[2^2]$ and any such orbit has a representative of the form
\[
e_{a,b,c}=\matr{
	0&0&0&0\\
	0&0&0&0\\
	b&a&0&0\\
	c&b&0&0
},\;a,b,c\in F.
\]
Denote by $v_i$ ($1\le i\le 4$) the $i$-th coordinate vector of our $4$-dimensional symplectic space.
The operator $e_{a,b,c}$ defines a non-degenerate quadratic form on $\span(v_1,v_2)$ by 
\begin{equation}\label{eq:quadform}
(X,Y)_{a,b,c}:=\langle X,e_{a,b,c}Y\rangle
\end{equation}
where $\langle\cdot,\cdot\rangle$ is as in (\ref{form1}). The $\Ad(G(F))$-orbit of $e_{a,b,c}$ is uniquely determined \cite[Prop. 5]{Nev11} by the isomorphism class of the quadratic form $(\cdot,\cdot)_{a,b,c}$. Similarly, a regular nilpotent $\Ad(G(F))$-orbit has a representative of the form
\begin{equation}\label{regnil}
n_d=\matr{
	0&0&0&0\\
	1&0&0&0\\
	0&d&0&0\\
	0&0&-1&0
},\; d\in F^{\times}.
\end{equation}
The orbit is again uniquely determined by the image of $d$ in $F^{\times}/(F^{\times})^2$. We show that
\begin{lemma}\label{lem:quad} The element $e_{a,b,c}$ lies in the analytic closure of $\Ad(G(F))n_d$ iff the quadratic form $(\cdot,\cdot)_{(a,b,c)}$ represents $d$, namely $(v,v)_{a,b,c}=d$ for some $v\in \span(v_1,v_2)$.
\end{lemma}

\begin{proof} Suppose $(v,v)_{a,b,c}=d$ for some $v\in\span(v_1,v_2)$. Then with a change of basis we may assume $(v_2,v_2)_{a,b,c}=d$, i.e. $a=d$. We have (with all hidden entries being $0$'s) for $e,f\in F$, $h\in F^{\times}$ that
\[
\matr{1&&&\\-e&1&&\\f&&1&\\&f&e&1}\matr{
	0&&&\\
	1&0&&\\
	&d&0&\\
	&&1&0
}\matr{1&&&\\e&1&&\\-f&&1&\\&-f&-e&1}=\matr{0&&&\\1&0&&\\de&d&0&\\2f+de^2&de&1&0}\]
\[
\matr{h^{-1}&&&\\&1&&\\&&1&\\&&&h}\matr{0&&&\\1&0&&\\de&d&0&\\2f+de^2&de&1&0}\matr{h&&&\\&1&&\\&&1&\\&&&h^{-1}}=\matr{0&&&\\h&0&&\\deh&d&0&\\2fh^2+de^2h^2&deh&h&0}.
\]
For arbitrarily small $h$ we can choose $e,f\in F$ so that $deh=b$, $2fh^2+de^2h^2=c$. Hence the above converges to $e_{a,b,c}$ as desired.
	
Now suppose $e_{a,b,c}$ is in the analytic closure of $\Ad(G(F))n_d$, i.e. there is a sequence $g_i\in G(F)$ such that $\Ad(g_i)^{-1}n_d$ converges to $e_{a,b,c}$. To show that $(\cdot,\cdot)_{a,b,c}$ represents $d$ we follow the method of Djokovi\'{c} \cite[Thm. 6]{Djo81} for real groups. The quadratic form
\[
	(X,Y)_{d,g_i}:=\langle X,\Ad(g_i)^{-1}(n_d)Y\rangle=\langle g_iX,n_dg_iY\rangle
\]
has to converge to $(X,Y)_{a,b,c}$ on $\span(v_1,v_2)$. Since being isomorphic to a non-degenerate quadratic form over $F$ is an open condition in the space of (not necessarily non-degenerate) quadratic forms, for $i\gg 0$ we have that $(\cdot,\cdot)_{d,g_i}|_{\span(v_1,v_2)}\cong(\cdot,\cdot)_{a,b,c}$. In particular, there exists a two-dimensional subspace $W$ in $F^4$ such that the restriction of the form $(X,Y)_d:=\langle X,n_dY\rangle$ is isomorphic to $(X,Y)_{a,b,c}$.

Observe the form $(X,Y)_d$ restricts to a rank $2$ hyperbolic form on $\span(v_1,v_3)$. The orthogonal complement of $\span(v_1,v_3)$ under it is $\span(v_2)\oplus\span(v_4)$, where the form has discriminant $d$ on $\span(v_2)$ and has $\span(v_4)$ in its kernel. The subspace $W$ must not intersect $\span(v_4)$, hence its image to $F^4/\span(v_4)\cong\span(v_1,v_2,v_3)$ is again $2$-dimensional. Denote by $W^{\perp}$ the orthogonal complement of $W$ in $\span(v_1,v_2,v_3)$. Since $(\cdot,\cdot)_d|_W\cong (\cdot,\cdot)_{a,b,c}$. We have that
\[ (\cdot,\cdot)_d|_{\span(v_1,v_3)}\oplus(\cdot,\cdot)_d|_{\span(v_2)}\cong(\cdot,\cdot)_d|_{\span(v_1,v_2,v_3)}\cong(\cdot,\cdot)_{a,b,c}\oplus(\cdot,\cdot)_{d}|_{W^{\perp}}\]
are isomorphic as quadratic spaces. Since $(\cdot,\cdot)_d|_{\span(v_1,v_3)}$ is hyperbolic, it is isomorphic to the direct sum of $(\cdot,\cdot)_{d}|_{W^{\perp}}$ and some other $1$-dimensional quadratic space. By the cancellation theorem of quadratic spaces \cite[pp. 34, Thm. 4]{Ser73}, we then have $(\cdot,\cdot)_{a,b,c}$ is isomorphic to the direct sum of $(\cdot,\cdot)_d|_{\span(v_2)}$ and this $1$-dimensional space, i.e. $(\cdot,\cdot)_{a,b,c}$ represents $d$, as asserted.
\end{proof}

\section{Shalika germs and proof of Theorem \ref{main}}

We normalize our Fourier transforms as
\[
\hat{f}(B):=\int_{\fg(F)}\psi(\Tr(AB))f(A)dA.
\]
where elements in $\fg(F)=\mathfrak{sp}_4(F)$ are identified as $4\times 4$ matrix as usual, i.e. as in (\ref{form1}). It is known (see the main result of \cite{KM03}, or \cite[(6.1)]{Ka15} for a more direct exhibition) that for any $\pi$ in Theorem \ref{main}, on some sufficiently small neighborhood $U$ of $0\in\fg(F)$ we have
\begin{equation}\label{Murnaghan}
(\Theta_{\pi}\circ\log^*)|_U\equiv c\cdot\hat{I}_A|_{U}
\end{equation}
for some $c\in\Q_{>0}$. 

Thanks to that our $p\ge 11$, the hypotheses needed for \cite[Thm 2.1.5]{De02a} are satisfied and it gives the following analogue of (\ref{LCE}), the {\bf Shalika germ expansion}:
\begin{equation}\label{Sha}
I_A(f)=\sum c_{\cO}(A)I_{\cO}(f).
\end{equation}
Here $\cO$ runs over nilpotent $\Ad(G(F))$-orbits in $\fg(F)$ as in (\ref{LCE}), and $f$ has to be a ``depth $-\frac{1}{2}$''-function; a condition that will be automatically met if $\hat{f}$ is supported in a small enough neighborhood.
Comparing (\ref{LCE}), (\ref{Murnaghan}) and (\ref{Sha}), we see that the coefficients in (\ref{LCE}) satisfy $c_{\cO}(\pi)=c\cdot c_{\cO}(A)$. In particular, $c_{\cO}(\pi)\not=0\Leftrightarrow c_{\cO}(A)\not=0$, and we have $\WF^{rat}(\pi)=\max\{\cO\;|\;c_{\cO}(A)\not=0\}$, where the partial order is given by (analytic) closure relation. Fix $\epsilon\in\cO_F^{\times}$ any non-square and
\begin{equation}\label{e}
	e=\matr{0&0&0&0\\0&0&0&0\\0&-\varpi^{-1}\epsilon&0&0\\\epsilon&0&0&0}.
\end{equation}
Theorem \ref{main} now follows from
\begin{proposition}\label{prop:reg} The Shalika germ $c_{\cO}(A)$ is zero for a regular nilpotent orbit $\cO$ iff the closure of $\cO$ contains $e$.
\end{proposition}

\begin{proposition}\label{prop:sub} The Shalika germ $c_{\cO}(A)$ is non-zero for the subregular nilpotent orbit $\cO=\Ad(G(F))e$.
\end{proposition}

\begin{remark} It might look like there are smart choices behind $e$ and $A$. In fact, a random choice of $A$ has about $\frac{1}{2}$ probability to work; it secretly needs certain invariant in $\cO_F^{\times}$ to be a square. Once that is met, Proposition \ref{prop:sub} will work for any such $A$ and some $e$ it picks out. Our choice merely gives nicer matrix calculation. The assumption $q\equiv1(\operatorname{mod}4)$ is also taken to simplify the exposition and is not essentially needed.
\end{remark}

The rest of the section is devoted to the proofs of Proposition \ref{prop:reg} and \ref{prop:sub}.

\begin{proof}[Proof of Proposition \ref{prop:reg}]
A result of Shelstad \cite{Sh89}, combined with another by Kottwitz \cite[Thm. 5.1]{Kot99} (we thank Alexander Bertoloni Meli for clarifying this), showed that for a regular nilpotent orbit $\cO$, $c_{\cO}(A)=0$ iff $\Ad(G(F))A$ does not meet the Kostant section associated to any element in $\cO$. The theory of Kostant section also tells that for any fixed regular $\cO$, among the stable orbit of $A$ there is exactly one rational orbit that meets the Kostant section. We have
\[
\Ad(\matr{1&0&0&0\\0&0&1&0\\0&-1&0&0\\0&0&0&1})A=\matr{0&0&0&\varpi^{-1}\\1&0&0&0\\0&-\varpi^{-1}&0&0\\0&0&-1&0}.
\]
is in the Kostant section for $n_{-\varpi^{-1}}$. Since $q\equiv 1(\operatorname{mod}4)$, we may fix $i:=\sqrt{-1}$ a square root of $-1$ in $\cO_F$. We have
\[
\Ad(\matr{1&i&0&0\\1&-i&0&0\\0&0&i/2&1/2\\0&0&-i/2&1/2})A=\matr{0&0&2\varpi^{-1}&0\\0&0&0&2\varpi^{-1}\\0&i/2&0&0\\-i/2&0&0&0}.
\]
Hence
\[
\Ad(\matr{2\varpi^{-1}&0&0&0\\0&1&0&0\\0&0&1&0\\0&0&0&\varpi/2}\matr{0&0&0&1\\0&1&0&0\\0&0&1&0\\-1&0&0&0}\cdot\matr{1&i&0&0\\1&-i&0&0\\0&0&i/2&1/2\\0&0&-i/2&1/2})A=\matr{0&0&0&2\varpi^{-2}i\\1&0&0&0\\0&i/2&0&0\\0&0&1&0}.
\]
is in the Kostant section for $n_{i/2}$. We note that both $-1$ and $i/2$ are squares in $\cO_F^{\times}$, and thus Lemma \ref{lem:quad} shows that $n_{-\varpi^{-1}}$ and $n_{i/2}$ are exactly the two regular nilpotent orbits whose closure does not contain $e$. This shows that if $c_{\cO}(A)=0$ for a regular nilpotent $\cO$, then the closure of $\cO$ must contain $e$. It remains to show that for any regular nilpotent orbit $\cO$ different from that of $n_{-\varpi^{-1}}$ and $n_{i/2}$, we have $c_{\cO}(A)=0$.

Let $\sqrt{\epsilon}$ be a square root of $\epsilon$ in an unramified quadratic extension of $F$. The element $d:=\matr{
\sqrt{\epsilon}^{-1}&0&0&0\\
0&\sqrt{\epsilon}^{-1}&0&0\\
0&0&\sqrt{\epsilon}&0\\
0&0&0&\sqrt{\epsilon}
}$ has image in $G_{ad}(F)=\PSp_4(F)$. Since the orbit of $A$ meets the Kostant section for $n_{-\varpi^{-1}}$ and $n_{i/2}$, the orbit of $\Ad(d)A$ meets the Kostant section of $\Ad(d)n_{-\varpi^{-1}}$ and $\Ad(d)n_{i/2}$. As $\Ad(d)n_{-\varpi^{-1}}=n_{-\epsilon \varpi^{-1}}$ and $\Ad(d)n_{i/2}=n_{\epsilon i/2}$ are the other two regular nilpotent orbits, using results of Shelstad and Kottwitz and the classical result that a Kostant section meets an $\Ad(G(\bar{F}))$-orbit at one point, it remains to prove that $\Ad(d)A$ and $A$ live in different $\Ad(G(F))$-orbits. The element $d$ defines a class $\alpha_d\in Z^1(F,Z(G))$ and the assertion that $\Ad(d)A$ and $A$ live in different $\Ad(G(F))$-orbits is equivalent to that the image of $\alpha_d$ in $H^1(F,Z_G(A))$ is non-trivial. Observe that $\alpha_d$ is trivial on inertia and sends $\Frob$ to $-1\in\mu_2=Z(G)$. Since $Z_G(A)$ is anisotropic over the maximal unramified extension of $F$, the image of $\alpha_d$ is non-trivial in $H^1(F,Z_G(A))=H^1(\Frob,X_*(Z_G(A))_{I_F})=H^1(\Frob,\mu_2^2)$, as claimed.
\end{proof}

\begin{proof}[Proof of Proposition \ref{prop:sub}]
Consider the characteristic function of the set
\[
\{X\in\fg(F)\;|\;X\text{ is of the form }
\matr{
\fm^0&\fm^0&\fm^{-1}&\fm^{-1}\\
\fm^0&\fm^0&\varpi^{-1}\epsilon+\fm^0&\fm^{-1}\\
\fm^0&\fm^0&\fm^0&\fm^0\\
\epsilon+\fm&\fm^0&\fm^0&\fm^0
}
\}
\]
Call this function $f$. It has the property that $f(X+Y)=f(X)$ whenever $Y$ is of the form
\[
\matr{
	\fm^0&\fm^0&\fm^{-1}&\fm^{-1}\\
	\fm^0&\fm^0&\fm^0&\fm^{-1}\\
	\fm^0&\fm^0&\fm^0&\fm^0\\
	\fm&\fm^0&\fm^0&\fm^0\\
}
\]
The set of elements of the above form is a Moy-Prasad lattice of depth $-\frac{1}{2}$. Since $A$ is of depth $-\frac{1}{2}$, \cite[Thm 2.1.5]{De02a} (or its application to Conj. 2 {\it op. cit}) shows that (\ref{Sha}) holds for $f$. Let $e$ be as in (\ref{e}). We claim that
\begin{lemma}\label{nil} Suppose $I_{\cO}(f)\not=0$ for a nilpotent $\Ad(G(F))$-orbit $\cO$. Then $e$ lies in the closure of $\cO$. 
\end{lemma}
\begin{lemma}\label{rs} $I_A(f)\not=0$.
\end{lemma}
With both lemmas, (\ref{Sha}) gives
$\sum_{\cO}c_{\cO}(A)I_{\cO}(f)=I_A(f)\not=0$. By Proposition \ref{prop:reg} and Lemma \ref{nil}, the only nilpotent orbit $\cO$ that can contribute to the sum is $\cO=\Ad(G(F))e$, which proves Proposition \ref{prop:sub}.
\end{proof}

\begin{proof}[Proof of Lemma \ref{nil}] We have
\[
\text{For }w=\matr{1&0&0&0\\0&0&-1&0\\0&1&0&0\\0&0&0&1},\;\Ad(w)\supp(f)=
\matr{
	\fm^0&\fm^0&\fm^{-1}&\fm^{-1}\\
	\fm^0&\fm^0&\fm^0&\fm^{-1}\\
	\fm^0&-\varpi^{-1}\epsilon+\fm^0&\fm^0&\fm^0\\
	\epsilon+\fm&\fm^0&\fm^0&\fm^0
}.
\]
For any $X\in\cO\cap\Ad(w)\supp(f)$, we observe that
\[
\varpi^{2n}\Ad(\matr{\varpi^n&0&0&0\\0&\varpi^n&0&0\\0&0&\varpi^{-n}&0\\0&0&0&\varpi^{-n}})X\in\cO\cap
\matr{
	\fm^{2n}&\fm^{2n}&\fm^{4n-1}&\fm^{4n-1}\\
	\fm^{2n}&\fm^{2n}&\fm^{4n}&\fm^{4n-1}\\
	\fm^0&-\varpi^{-1}\epsilon+\fm^0&\fm^{2n}&\fm^{2n}\\
	\epsilon+\fm&\fm^0&\fm^{2n}&\fm^{2n}
}
\]
As $n$ goes to $+\infty$, such elements converge (or a subsequence does) to an element
\[
e'\in
\matr{
	0&0&0&0\\
	0&0&0&0\\
	\fm^0&-\varpi^{-1}\epsilon+\fm^0&0&0\\
	\epsilon+\fm&\fm^0&0&0
}.
\]
This element $e'$ lives in the same $\Ad(G(F))$-orbit as $e$ because the bottom-left $2\times 2$ matrix defines an isomorphic quadratic form. This proves the lemma.
\end{proof} 

\begin{proof}[Proof of Lemma \ref{rs}]
We look at
\[
\Ad(\matr{1&0&0&0\\z&1&0&0\\0&0&1&0\\0&0&-z&1})A=
\matr{
	0&0&\varpi^{-1}z&\varpi^{-1}\\
	0&0&\varpi^{-1}(1+z^2)&\varpi^{-1}z\\
	1&0&0&0\\
	-2z&1&0&0
}
\]
and
\[
\Ad(\matr{x&0&0&0\\0&y^{-1}&0&0\\0&0&y&0\\0&0&0&x^{-1}}\cdot\matr{1&0&0&0\\z&1&0&0\\0&0&1&0\\0&0&-z&1})A=
\matr{
	0&0&\varpi^{-1}xy^{-1}z&\varpi^{-1}x^2\\
	0&0&\varpi^{-1}y^{-2}(1+z^2)&\varpi^{-1}xy^{-1}z\\
	x^{-1}y&0&0&0\\
	-2x^{-2}z&x^{-1}y&0&0
}
\]
Suppose $x,y\in\cO_F^{\times}$ and $z\in\cO_F$. Denote by $\bar{x},\bar{y},\bar{z},\bar{\epsilon}\in\F_q$ the respective reductions. The right hand side of the last equation lies in the support of $f$ iff
\[
\left\{\begin{array}{l}-2\bar{x}^{-2}\bar{z}=\bar{\epsilon}\\\bar{y}^{-2}(1+\bar{z}^2)=\bar{\epsilon}
\end{array}\right.\Leftrightarrow\left\{\begin{array}{l}\bar{z}=-\bar{\epsilon}\bar{x}^2/2\\\frac{\bar{\epsilon}^2}{4}\bar{x}^4+1=\bar{\epsilon}\bar{y}^2
\end{array}\right.
\]
That is, as long as the curve $E=(\bar{\epsilon}\bar{y}^2=\frac{\bar{\epsilon}^2}{4}\bar{x}^4+1\;|\;(\bar{x},\bar{y})\in(\mathbb{G}_m)^2/_{\F_q})$ has an $\F_q$-point, there exists $g\in G(F)$ such that $\Ad(g)A\in\supp(f)$, i.e. $I_A(f)\not=0$. Such an $\F_q$-point always exists. Indeed, $E$ differs from its smooth completion $E^c$ by eight $\bar{\F}_q$-points (two for $\bar{x}=0$, four for $\bar{y}=0$ and two at infinity), and none of them is defined over $\F_q$ because $\bar{\epsilon}\in\F_q$ is a non-square. Hence $E(\F_q)=E^c(\F_q)$, while $E^c$ is a geometrically connected projective smooth genus $1$ curve and always have an $\F_q$-point.
\end{proof}

\begin{remark} This ``none of the boundary points is defined over the residue field'' phenomenon seems to be related to the vanishing of $c_{\cO}(A)$ for those $\cO>\Ad(G(F))e$.
\end{remark}

\begin{remark}
Using a special case of a theorem of Kim and Murnaghan \cite[Thm. 2.3.1]{KM03} that $\Ad(g)A\in\fg(F)_{-\frac{1}{2}}\implies g\in G(F)_0$, one may reduce the computation of orbital integrals and thus $c_{\cO}(A)$ and $c_{\cO}(\pi)$ (for $\cO=\Ad(G(F))e$) to $\#E(\F_q)$. We predict that the dimension of the associated degenerate Whittaker model to be $\frac{1}{4}\#E(\F_q)$, analogous to \cite[Thm. 4.10 and Cor. 6.2]{Ts17}.
\end{remark}

\begin{remark} We may also work with depth $n+\frac{1}{2}$ representations, by replacing $A$ by $\varpi^{-n}A$ and replacing $G(F)_{\frac{1}{2}}$ by $G(F)_{n+\frac{1}{2}}$ in Theorem \ref{main}. The same proof works, except that $e$ needs to be replaced by $\varpi^{-n}e$, resulting in that every $\cO\in\WF^{rat}(\pi)$ is replaced by $\varpi^{-n}\cO$.
\end{remark}

\section{Langlands parameters}

The determination of the Langlands parameter corresponding to an individual $\pi$ in Theorem \ref{main} is part of the difficult problem solved in \cite{Ka15} with deep insight on the rectifying characters and their relation with transfer factors. The {\it collection} of all Langlands parameters corresponding to components $\pi$ in Theorem \ref{main} is nevertheless simpler, because it happens in this case that the rectifying characters can be absorbed into the choice of an irreducible component in $\operatorname{c-ind}_{G(F)_{\frac{1}{2}}}^{G(F)}\til{\psi}_{A}$. We describe the collection of such Langlands parameters, in the hope that it may be useful to interested readers.

Consider the ramified quadratic extension $E=F(\sqrt{\varpi})$. Write $\varpi_E=\sqrt{\varpi}$. A Langlands parameter we seek for is a homomorphism $\rho: W_F\ra \mathrm{SO}_5(\C)$. It has image in $\mathrm{O}_2(\C)\times\mathrm{O}_2(\C)\times\mathrm{SO}_1(\C)$, i.e. $\rho$ can be viewed as the sum of two orthogonal self-dual representations and a trivial representation. We have $\rho=\rho_1\oplus\rho_2\oplus\operatorname{triv}$ where $\rho_j=\Ind _{W_E}^{W_F}\chi_j$ for $j=1,2$. Write $\alpha_1=1$ and $\alpha_2=\sqrt{-1}$ for any choice of square root of $-1$ in $\cO_F^{\times}$. Then $\chi_j$ is a character on $E^{\times}$ satisfying
\begin{enumerate}[label=(\alph*)]
	\item $\chi_j|_{F^{\times}}\equiv 1$.
	\item $\chi_j(1+x\varpi)=1$ for all $x\in\cO_E$.
	\item $\chi_j(1+x\varpi_E)=\psi(2x\alpha_j)$ for all $x\in\cO_E$.
\end{enumerate}
Here $\psi$ is as chosen before (\ref{A}). We note that each $\chi_j$ is determined up to a freedom of $\chi_j(\varpi_E)\in\{\pm1\}$, and consequently there are $2^2=4$ candidates for such $\rho$. Relatedly, there are also $2^2$ components of $\pi$ in Theorem \ref{main}.

\p
\bibliographystyle{amsalpha}
\bibliography{biblio.bib}

\end{document}